\documentclass[11pt]{amsart}
\usepackage[T1]{fontenc}
\usepackage{lmodern}
\usepackage[margin=1.12in]{geometry}
\usepackage{amsmath,amssymb,amsthm,mathtools}
\usepackage{microtype}
\usepackage[colorlinks=true,citecolor=blue,linkcolor=red,urlcolor=blue]{hyperref}
\hypersetup{
 pdftitle={Numerical Kawamata--Viehweg vanishing for rank three matroids},
 pdfauthor={Ronnie Cheng}
}

\numberwithin{equation}{section}
\newtheorem{theorem}{Theorem}[section]

\newtheorem{lemma}[theorem]{Lemma}
\newtheorem{corollary}[theorem]{Corollary}
\newtheorem{conjecture}[theorem]{Conjecture}
\newtheorem{mainthm}{Theorem}

\theoremstyle{definition}
\newtheorem{definition}[theorem]{Definition}
\newtheorem{remark}[theorem]{Remark}

\newcommand{\Z}{\mathbb Z}
\newcommand{\Q}{\mathbb Q}
\newcommand{\R}{\mathbb R}
\newcommand{\C}{\mathbb C}
\newcommand{\PP}{\mathbb P}
\newcommand{\OO}{\mathcal O}
\newcommand{\LL}{\mathcal L}
\newcommand{\FF}{\mathcal F}
\newcommand{\cH}{\mathcal H}
\newcommand{\cP}{\mathcal P}
\newcommand{\degM}{\deg_M}
\newcommand{\invchi}[2]{\chi_{#1}([\LL_{#2}^{-1}])}
\DeclareMathOperator{\rk}{rk}
\DeclareMathOperator{\Bl}{Bl}

\title[Numerical vanishing in rank three]{Numerical Kawamata--Viehweg vanishing\\ for rank-three matroids}
\author{Ronnie Cheng}
\address{Department of Mathematics, Stanford University, USA}
\email{rtcheng@stanford.edu}
\date{October 29, 2026}

\begin{document}
\begin{abstract}
We prove the numerical Kawamata--Viehweg conjecture for every matroid of rank three: the Euler characteristic of the inverse of an integral big and nef line bundle class is nonnegative. The proof reduces the conjecture to an inequality for a stone game.
\end{abstract}
\maketitle
\section{Introduction}\label{sec:intro}
Let $U$ be a vector space of dimension $r$ over a field $k$, and let $(\ell_e)_{e\in E}$ be a nonempty finite family of linear forms spanning $U^\vee$. They define a matroid $M$ of rank $r$ on $E$ by
\[
 \rk_M(S)=\dim_k\operatorname{span}\{\ell_e:e\in S\},
 \qquad S\subseteq E.
\]
A matroid is \emph{realizable} over $k$ if it arises in this way. The injective map $u\mapsto(\ell_e(u))_{e\in E}$ identifies $U$ with a subspace $L\subseteq k^E$, called a \emph{realization} of $M$. From now on, all matroids are assumed loopless; in a realization, this means $\ell_e\ne0$ for every $e\in E$. 

After identifying $U$ with $L$, the hyperplanes $H_e=\PP(\ker\ell_e)\subset\PP(L)$ are the restrictions of the coordinate hyperplanes in $\PP(k^E)$. The arrangement complement $\PP(L)\setminus\bigcup_{e\in E}H_e$ admits a smooth projective de Concini--Procesi wonderful compactification $W_L$ of dimension $d=r-1$ \cite{DCP}. We use the maximal wonderful compactification, obtained from $\PP(L)$ by successively blowing up the strict transforms of all nonempty proper intersection subspaces in increasing order of dimension. We write $\pi:W_L\to\PP(L)$ for the resulting morphism. In particular, for a realization of a rank-three matroid, $W_L$ is the blowup of $\PP^2_k$ at the intersection points corresponding to rank-two flats.

Although different realizations of $M$ can yield nonisomorphic varieties $W_L$, their Chow rings admit a presentation depending only on the lattice of flats \cite{DCP,FY}. For an arbitrary matroid $M$, we use the following presentation of its \emph{Chow ring}, with integral coefficients \cite[Definition~1.3]{AHK}:
\begin{equation}\label{eq:chow-presentation}
 A^*(M)=\Z[x_F:\varnothing\ne F\subsetneq E\text{ a flat}]/(I+J),
\end{equation}
where $I$ is generated by $x_Fx_G$ for incomparable flats $F$ and $G$, and $J$ is generated by $\sum_{F\ni i}x_F-\sum_{F\ni j}x_F$ for $i,j\in E$. Recall that a flat of $M$ is a subset $F\subseteq E$ such that adding any element of $E\setminus F$ increases its rank. 
For every realization $L$ and every nonempty proper flat $F$, let $E_F$ denote the divisor on $W_L$ obtained from the exceptional divisor when the strict transform of $\bigcap_{e\in F}H_e$ is blown up. In particular, when $\rk F=1$, this blowup is an isomorphism and $E_F$ is simply the strict transform of the corresponding hyperplane. Then $x_F\mapsto[E_F]$ gives an isomorphism
\[
 A^*(M)\simeq A^*(W_L).
\]
The degree map $\degM:A^d(M)\to\Z$ sends each complete flag monomial to $1\in\Z$ and agrees with the geometric degree map. The class
\begin{equation}\label{eq:alpha}
 \alpha=\sum_{F\ni e}x_F
\end{equation}
is independent of $e\in E$ and corresponds to the pullback under $\pi$ of the hyperplane class on $\PP(L)$.

Huh--Katz used this intersection theory to prove log-concavity of the characteristic polynomial for matroids realizable over any field \cite{HK}. Adiprasito--Huh--Katz then proved that $A^*(M)$ satisfies Poincar\'e duality, hard Lefschetz, and the Hodge--Riemann relations for every matroid, with no realizability assumption \cite{AHK}. Thus the ring retains the ``K\"{a}hler package'' of a smooth projective variety even when there is no underlying arrangement. 

The same progression from geometry to combinatorics underlies matroid $K$-theory. Building on the tautological classes and $K$-theoretic constructions of Berget--Eur--Spink--Tseng \cite{BEST}, Larson--Li--Payne--Proudfoot gave presentations of $K^0(W_L)$ and formulas for its Euler characteristic depending only on $M$. These define a \emph{$K$-ring} $K(M)$ and an \emph{Euler characteristic} $\chi_M:K(M)\to\Z$ for every matroid \cite{LLPP}. A divisor class $D\in A^1(M)$ determines a line bundle class $[\LL_D]\in K(M)$, and for every realization $L$ we have
\begin{equation}\label{eq:geometric-K}
 K(M)\simeq K^0(W_L),\qquad
 \chi_M([\LL_D])=\chi(W_L,\LL_D),
\end{equation}
where on the right $\LL_D$ is the line bundle on $W_L$ with first Chern class $D$. 

In \cite{Cheng}, the author constructed a tangent class $T_M\in K(M)$ that recovers $[T_{W_L}]$ for every realization $L$. Its Todd class satisfies the Hirzebruch--Riemann--Roch identity \cite[Corollary~3.20]{Cheng}
\begin{equation}\label{eq:matroid-HRR}
 \chi_M(\xi)=\degM\bigl(\operatorname{ch}(\xi)\operatorname{td}(T_M)\bigr),
 \qquad \xi\in K(M),
\end{equation}
where $\operatorname{ch}$ is the usual Chern character and $\degM$ is applied to the component of degree $d$ in $A^*(M)_\Q$.

\subsection{Positivity and the numerical vanishing problem}\label{sec:positivity}
The notion of combinatorial nefness below comes from \cite[Definitions~4.1--4.2]{AHK}, expressed in terms of the generators $x_F$. A flag of flats is a chain of flats under inclusion.
\begin{definition}\label{def:combinatorial-positivity}
A class $D\in A^1(M)_\R$ is \emph{combinatorially nef} if, for every flag $\FF$ of nonempty proper flats, including the empty flag, it admits an expression
\[
 D=\sum_G c_Gx_G
\]
with real coefficients such that $c_F=0$ for $F\in\FF$ and $c_G\ge0$ whenever $G\notin\FF$ and $\FF\cup\{G\}$ is a flag. 
\end{definition}

For every realization $L$, combinatorially nef classes are geometrically nef on $W_L$ \cite[Proposition~5.5(2)]{Cheng}. Following \cite[Definition~5.7]{Cheng}, we call a combinatorially nef class $D$ \emph{big} if $\degM(D^d)>0$. Consequently, a combinatorially big and nef class is also geometrically big and nef.

In characteristic zero, Kawamata--Viehweg vanishing \cite{Lazarsfeld} gives
\[
 (-1)^d\chi(W_L,\LL_D^{-1})
 =h^d(W_L,\LL_D^{-1})\ge0
\]
for every big and nef line bundle $\LL_D$. In view of \eqref{eq:geometric-K}, this motivates a statement involving only the matroid.
\begin{conjecture}[Numerical matroid Kawamata--Viehweg; {\cite[Conjecture 5.16]{Cheng}}]\label{conj:KV}
For every matroid $M$ of rank $r$ and every combinatorially big and nef divisor $D\in A^1(M)$,
\[
 (-1)^{r-1}\invchi{M}{D}\ge0.
\]
\end{conjecture}

\subsection{Main result and proof strategy}
We prove Conjecture~\ref{conj:KV} for all matroids of rank three.

\begin{mainthm}\label{thm:rank3}
Let $M$ be a loopless matroid of rank three. If $D\in A^1(M)$ is combinatorially big and nef, then
\[
 \invchi{M}{D}\ge0.
\]
\end{mainthm}

The theorem applies to every integral big and nef class in the combinatorial nef cone, with no realizability assumption on $M$.

Write
\[
 D=a\alpha-\sum_{H\in\cH}b_Hx_H,
\]
where $\cH$ is the set of rank-two flats and $\alpha$ is the class defined in \eqref{eq:alpha}. It can be shown that $a$ and $b_H$ are nonnegative when $D$ is nef. The rank-three intersection pairing and Riemann--Roch formula give
\[
 \degM(D^2)=a^2-\sum_Hb_H^2,
 \qquad
 \invchi{M}{D}=\binom{a-1}{2}-\sum_H\binom{b_H}{2}.
\]
Thus the required numerical vanishing is equivalent to
\begin{equation}\label{eq:rank3-intro}
 \sum_Hb_H(b_H-1)\le(a-1)(a-2).
\end{equation}

We prove the theorem by encoding $a$ and the sequence $b_H$ as a legal play of a stone game, whose rules and score bound are stated in Theorem~\ref{thm:stone-game}. The positive coefficients $b_H$ become the move sizes, and bigness says that at least one stone remains. The score bound then gives Theorem~\ref{thm:rank3}.

To prove Theorem~\ref{thm:stone-game}, we construct a complex surface $S$ that is a blowup of $\PP_{\C}^2$ at points on a smooth cubic. Consider the divisor $D = a\alpha - \sum b_HE_H$ on $S$, where $\alpha$ is the pullback of the hyperplane class and $E_H$ are the exceptional divisors. After an immediate case is excluded, the divisor $D$ is big and nef, and classical Kawamata--Viehweg vanishing on $S$ proves the required inequality. Remark~\ref{rem:cremona-proof} outlines a combinatorial proof of the game inequality motivated by Cremona transformations.

As a consequence of Theorem~\ref{thm:rank3}, we also obtain monotonicity in nef directions. If $D$ is big and nef and $E$ is nef, then
\[
 \invchi{M}{D+E}\ge\invchi{M}{D};
\]
see Corollary~\ref{cor:monotonicity}.

\subsection{Relation to earlier work and organization}
Eur--Larson prove $K$-theoretic positivity for nonnegative integral combinations of simplicial line bundles \cite[Theorem~1.5]{EL}. Eur--Fink--Larson prove positivity results for divisor classes induced by nef divisors on the permutohedral variety \cite[Theorem~B]{EFL}. These results imply the numerical Kawamata--Viehweg inequality for the corresponding classes when their top self-intersection is positive. Theorem~\ref{thm:rank3} proves the numerical inequality for every combinatorially big and nef class in rank three.

Section~\ref{sec:rank3} reduces Theorem~\ref{thm:rank3} to the stone game (Theorem~\ref{thm:stone-game}) and deduces monotonicity in nef directions. Section~\ref{sec:stone-game} proves Theorem~\ref{thm:stone-game}.

\subsection{Acknowledgement}
The author would like to thank Ting-Wei Chao for sharing the proof of the subcase $m=1$ of Theorem~\ref{thm:stone-game}. 
\subsection{Use of AI}
This work was developed with substantial assistance from ChatGPT.
Its use included proposing proof strategies and revising the exposition. The author takes full responsibility for the final manuscript.
\section{Reduction to a stone game}\label{sec:rank3}

Throughout this section, $M$ is a loopless matroid of rank three on a finite ground set $E$. All divisor classes in the vanishing and monotonicity statements are integral.

Nefness is as in Definition~\ref{def:combinatorial-positivity}, and we write \emph{nef} for combinatorially nef. The Chow ring, degree map, and class $\alpha$ are those introduced in Section~\ref{sec:intro}.

\subsection{Rank-three coordinates and Euler characteristic}

Write $\cP$ and $\cH$ for its flats of ranks one and two, respectively. The Chow relations give
\[
 x_p=\alpha-\sum_{H\supset p}x_H
 \qquad(p\in\cP),
\]
so $\alpha$ and the classes $x_H$, for $H\in\cH$, span $A^1(M)$. They satisfy
\begin{equation}\label{eq:rank-three-pairing}
 \degM(\alpha^2)=1,\qquad
 \alpha x_H=0,\qquad
 \degM(x_H^2)=-1,\qquad
 x_Hx_J=0\quad(H\ne J).
\end{equation}
These classes form an integral basis. We henceforth identify products of two divisor classes with their degrees. For $D\in A^1(M)$, write
\[
 D=a\alpha-\sum_{H\in\cH}b_Hx_H,
 \qquad a,b_H\in\Z.
\]
Then
\begin{equation}\label{eq:surface-degrees}
 D^2=a^2-\sum_Hb_H^2,\qquad
 Dx_H=b_H,\qquad
 Dx_p=a-\sum_{H\supset p}b_H.
\end{equation}

If $D$ is nef, then $Dx_F\ge0$ for every nonempty proper flat $F$. Indeed, apply Definition~\ref{def:combinatorial-positivity} to the flag $\{F\}$. The resulting expression has no $x_F$ term and has nonnegative coefficients at the other flats comparable to $F$. Multiplication by $x_F$ kills the incomparable terms and leaves a nonnegative linear combination of monomials corresponding to complete flags. Consequently,
\begin{equation}\label{eq:nef-bounds}
 b_H\ge0,\qquad
 \sum_{H\supset p}b_H\le a\quad(p\in\cP).
\end{equation}
These inequalities imply $a\ge0$. If $a=0$, then every $b_H$ vanishes, since every flat of rank two contains a flat of rank one. Thus every nonzero nef class has $a>0$ and $b_H\ge0$, as in \cite[Lemma~5.26]{Cheng}.

In rank three, the canonical class is
\[
 K_M=-c_1(T_M)=-3\alpha+\sum_Hx_H
\]
by \cite[Corollary~3.22]{Cheng}. The Hirzebruch--Riemann--Roch formula \eqref{eq:matroid-HRR} gives
\begin{equation}\label{eq:surface-RR}
 \chi_M([\LL_D])=1+\frac{D(D-K_M)}2.
\end{equation}
In particular, using $[\LL_{-D}]=[\LL_D^{-1}]$,
\begin{equation}\label{eq:surface-Euler}
 \invchi{M}{D}
 =1+\frac{D(D+K_M)}2
 =\binom{a-1}{2}-\sum_H\binom{b_H}{2}.
\end{equation}

\subsection{From a nef divisor to a legal game}

The following statement is independent of matroids. We will prove it in Section~\ref{sec:stone-game}.

\begin{theorem}[The stone game]\label{thm:stone-game}
Let $a,m$ be positive integers. Start with $ma$ piles, each containing $a$ stones. A move of positive integer size $b$ chooses $mb$ distinct piles containing at least $b$ stones each, removes $b$ stones from each chosen pile, and earns $b(b-1)$ points. If at least one stone remains after a finite sequence of legal moves, the total score is at most
\[
 (a-1)(a-2).
\]
Equivalently, if a legal sequence has move sizes $b_1,\ldots,b_s$ satisfying $\sum_i b_i^2<a^2$, then
\begin{equation}\label{eq:stone-game-bound}
 \sum_i b_i(b_i-1)\le(a-1)(a-2).
\end{equation}
\end{theorem}

For $a\ge2$, a single move of size $a-1$ attains the bound and leaves $m(2a-1)$ stones. For $a=1$, the empty sequence attains the bound.

\begin{proof}[Proof of Theorem~\ref{thm:rank3}, assuming Theorem~\ref{thm:stone-game}]
Let $D=a\alpha-\sum_Hb_Hx_H$ be big and nef. Nefness for the empty flag expresses $D$ as a nonnegative linear combination of the classes $x_F$. The defining linear equations have rational coefficients, so the coefficients may be chosen rational. Clearing denominators gives
\[
 mD=\sum_p n_px_p+\sum_H n_Hx_H,
 \qquad m\in\Z_{>0},\quad n_p,n_H\in\Z_{\ge0}.
\]
Intersecting with $\alpha$ and $x_H$, respectively, yields
\begin{equation}\label{eq:pile-counts}
 \sum_p n_p=ma,\qquad
 \sum_{p\subset H}n_p=mb_H+n_H\ge mb_H.
\end{equation}

Start with $n_p$ piles labelled by each $p\in\cP$, each containing $a$ stones. For every $H$ with $b_H>0$, choose $mb_H$ piles whose labels are contained in $H$, and make a move of size $b_H$ on these piles. There are enough such piles by \eqref{eq:pile-counts}. A pile labelled by $p$ loses at most $\sum_{H\supset p}b_H\le a$ stones in total, by \eqref{eq:nef-bounds}. Hence all these moves are legal, in any order. They leave
\[
 ma^2-m\sum_Hb_H^2=mD^2>0
\]
stones. Theorem~\ref{thm:stone-game} therefore gives
\[
 \sum_Hb_H(b_H-1)\le(a-1)(a-2),
\]
and \eqref{eq:surface-Euler} implies $\invchi{M}{D}\ge0$.
\end{proof}

\begin{corollary}\label{cor:monotonicity}
Let $M$ have rank three, and let $D,E\in A^1(M)$. If $D$ is big and nef and $E$ is nef, then
\[
 \invchi{M}{D+E}\ge\invchi{M}{D}.
\]
\end{corollary}
\begin{proof}
An expression for the empty flag $E=\sum_F c_Fx_F$, with $c_F\ge0$, gives $E^2\ge0$ and $DE\ge0$, since $D$ and $E$ intersect every $x_F$ nonnegatively.

If $E^2>0$, Cauchy--Schwarz in the coordinates of \eqref{eq:rank-three-pairing} gives $DE\ge\sqrt{D^2E^2}\ge1$. Riemann--Roch \eqref{eq:surface-RR} and Theorem~\ref{thm:rank3} now give
\[
 \invchi{M}{D+E}-\invchi{M}{D}
 =\invchi{M}{E}-1+DE\ge0.
\]
If $E^2=0$, then $D+tE$ is big and nef for every integer $t\ge0$, since
\[
 (D+tE)^2=D^2+2tDE>0.
\]
By Riemann--Roch, $\invchi{M}{D+tE}$ is a polynomial of degree at most one in $t$. Theorem~\ref{thm:rank3} makes it nonnegative for every integer $t\ge0$, so its slope is nonnegative. Taking $t=1$ finishes the proof.
\end{proof}

\section{Proof of the stone game}\label{sec:stone-game}
We begin by extracting a family of integer inequalities from the legal moves.

\subsection{Integer tests}
Let $S$ be the blowup of $\PP^2_{\C}$ at $s$ distinct points, with hyperplane class $H$ and exceptional curves $E_i$. For
\[
 A=aH-\sum_i b_iE_i,\qquad
 B=hH-\sum_i u_iE_i,\qquad h,u_i\in\Z_{\ge0},
\]
Riemann--Roch gives
\[
 \chi(S,\OO_S(B))
 =1+\frac{B(B-K_S)}2
 =1+\frac{h(h+3)-\sum_i u_i(u_i+1)}2.
\]
A degree-$h$ plane form has $\binom{h+2}{2}$ coefficients, and multiplicity at least $u_i$ at a prescribed point imposes at most $\binom{u_i+1}{2}$ linear conditions. Thus $\chi(S,\OO_S(B))\ge1$ guarantees such a nonzero form, so $B$ is linearly equivalent to an effective divisor. In particular, nefness of $A$ would imply
\[
 \chi(S,\OO_S(B))\ge1\quad\Longrightarrow\quad|B|\ne\varnothing \quad \Longrightarrow\quad AB\ge0.
\]
Since $AB=ah-\sum_i b_i u_i$, the following lemma captures the numerical ``nefness'' of the stone game.

\begin{lemma}\label{lem:integer-tests}
For any legal sequence of moves of sizes $b_1,\ldots,b_s$ in the stone game, and any nonnegative integers $h,u_1,\ldots,u_s$,
\begin{equation}\label{eq:integer-tests}
 \sum_i u_i(u_i+1)\le h(h+3)
 \quad\Longrightarrow\quad
 \sum_i b_i u_i\le ah.
\end{equation}
\end{lemma}
\begin{proof}

We argue by induction on $h$, the case $h=0$ is immediate. Suppose that $h,u_1,\ldots,u_s$ satisfy the hypothesis of \eqref{eq:integer-tests} but $\sum_i b_i u_i>ah$. Let $S_i$ be the set of piles chosen by move $i$. Thus $|S_i|=mb_i$, and the legality of the stone game gives
\[
 \sum_{i:j\in S_i}b_i\le a
 \qquad\text{for every pile }j.
\] Then
\[
 \sum_{j=1}^{ma}\sum_{i:j\in S_i}u_i
 =m\sum_i b_i u_i>mah.
\]
Hence some pile $j$ satisfies $\sum_{i:j\in S_i}u_i\ge h+1$. Decrease every positive $u_i$ with $j\in S_i$ by one, obtaining $u'_i\ge0$. We have
\[
 \sum_i u_i(u_i+1)-\sum_i u'_i(u'_i+1)
 =2\sum_{i:j\in S_i}u_i\ge2(h+1).
\]
Consequently,
\[
 \sum_i u'_i(u'_i+1)
 \le h(h+3)-2(h+1)=(h-1)(h+2),
\]
whereas
\[
 \sum_i b_i u'_i
 =\sum_i b_i u_i-\sum_{\substack{i:j\in S_i\\u_i>0}}b_i
 >ah-a=a(h-1).
\]
This contradicts the induction hypothesis.
\end{proof}
\subsection{An auxiliary complex surface}

When a rank-three matroid is realized by a line arrangement in $\PP^2_{\C}$, the corresponding wonderful compactification is the blowup of $\PP^2_{\C}$ at the intersection points, and the numerical inequality follows from classical Kawamata--Viehweg vanishing. 

The intersection pairing \eqref{eq:rank-three-pairing} does not depend on the positions of the points. Thus, on any blowup of $\PP^2_{\C}$ at $s$ distinct points, the self-intersection of $A=aH-\sum_i b_iE_i$ and the Euler characteristic $\chi(S,\OO_S(-A))$ are given by the same formulas. After excluding the immediate case handled below, our strategy is to find a special blowup on which $A$ is big and nef, and then apply Kawamata--Viehweg vanishing. 

The auxiliary surface that we construct is a blowup of $\PP^2_{\C}$ at very general points on a fixed smooth cubic in the setting studied by Harbourne \cite{Harbourne}. For a description of negative curves on such surfaces, see \cite[Theorem~2.4]{generalpointsoncubic}. On this special blowup, every integral curve $B$ other than the exceptional curves and the strict transform of the cubic satisfies $\chi(S,\OO_S(B))\ge1$. Thus Lemma~\ref{lem:integer-tests}, together with the intersections with those distinguished curves, will suffice to prove nefness.

\begin{proof}[Proof of Theorem~\ref{thm:stone-game}]
Let $b_1,\ldots,b_s$ be the move sizes. Fix a smooth plane cubic $C\subset\PP^2_{\C}$ and choose $P_1,\ldots,P_s$ very general on $C$. Then
\begin{equation}\label{eq:independence}
 \OO_C(h)\not\simeq\OO_C\!\left(\sum_i u_iP_i\right)
 \quad\text{whenever }h\in\Z_{\ge1},\ u_i\in\Z_{\ge0}.
\end{equation}

Let $S=\Bl_{P_1,\ldots,P_s}\PP^2_{\C}$ and $A=aH-\sum_i b_iE_i$, with notation as above, and let $\widetilde C$ be the strict transform of $C$. Then
\[
 K_S=-\widetilde C,\qquad
 A^2=a^2-\sum_i b_i^2>0.
\]
Riemann--Roch gives
\begin{equation}\label{eq:stone-RR}
 \chi(S,\OO_S(-A))
 =1+\frac{A^2-A\widetilde C}{2}
 =\binom{a-1}{2}-\sum_i\binom{b_i}{2}.
\end{equation}
Our goal is to show that this Euler characteristic is nonnegative. If $A\widetilde C\le0$, it is positive because $A^2>0$. We may therefore assume $A\widetilde C>0$.

We claim that $A$ is nef. Its intersections with the exceptional curves are $AE_i=b_i\ge0$, and $A\widetilde C>0$. Let $B\subset S$ be any other integral curve. It has class $hH-\sum_i u_iE_i$ with $h\ge1$ and $u_i\ge0$. Since $B$ and $\widetilde C$ are distinct, $B\widetilde C\ge0$. Equality would force the intersection divisor of the plane image of $B$ with $C$ to be $\sum_i u_iP_i$, contrary to \eqref{eq:independence}. Hence $B\widetilde C>0$. 
Writing $p_a(B)\ge0$ for the arithmetic genus of $B$, adjunction and Riemann--Roch give
\[
 \chi(S,\OO_S(B))
 =1+\frac{B(B-K_S)}2
 =p_a(B)-BK_S
 =p_a(B)+B\widetilde C>0.
\]
Lemma~\ref{lem:integer-tests} now gives $AB\ge0$, proving the claim.

Since $A^2>0$, the nef divisor $A$ is big. Classical Kawamata--Viehweg vanishing \cite{Kawamata,Viehweg}, followed by Serre duality, gives
\[
 \chi(S,\OO_S(-A))=h^2(S,\OO_S(-A))\ge0.
\]
By \eqref{eq:stone-RR}, this is the desired inequality.
\end{proof}

\begin{remark}[Outline of a combinatorial proof]
\label{rem:cremona-proof}
The integer tests also lead to a combinatorial proof motivated by quadratic Cremona transformations. After sorting the $b_i$ in nonincreasing order and adjoining zeros if necessary, whenever $b_1+b_2+b_3>a$, replace
\[
 (a;b_1,b_2,b_3)\longmapsto
 (2a-b_1-b_2-b_3;\ a-b_2-b_3,\ a-b_1-b_3,\ a-b_1-b_2),
\]
and leave the other entries unchanged. This preserves the tests in \eqref{eq:integer-tests}, the self-intersection $a^2-\sum_i b_i^2$, and the Euler characteristic in \eqref{eq:stone-RR}, while decreasing $a$. After reordering at each step, one reaches the terminal case $b_1+b_2+b_3\le a$, where the desired bound is an elementary inequality. These transformations preserve the numerical tests; the transformed tuple need not itself arise from a legal game.
\end{remark}

\bibliographystyle{amsalpha}
\bibliography{bibliography}
\end{document}